\documentclass[10pt,leqno]{amsart}
\usepackage{graphicx}
\usepackage{indentfirst,csquotes}

\usepackage[english]{babel}
\usepackage{fancyhdr} 
\usepackage{subfig}

\usepackage{hyperref}
\usepackage{amssymb}
\usepackage{mathtools}
\usepackage{amsmath}
\usepackage{latexsym}
\usepackage{amsthm}
\usepackage{esint}
\usepackage{wasysym}

\usepackage{etoolbox}
\usepackage[ruled]{algorithm2e}
\usepackage{algorithmic}
\usepackage{longtable}
\usepackage{pst-eucl}
\usepackage{comment}
\usepackage{hyperref}
\hypersetup{
    colorlinks=true,
    linkcolor=blue,
    citecolor=teal
    }
\let\div\undefined\DeclareMathOperator{\div}{div}
\newcommand{\divG}[0]{\div_{\mathbb{G}}}
\newcommand{\DeltaG}[0]{\Delta_{\mathbb{G}}}
\newcommand{\GammaG}[0]{\Gamma_{\mathbb{G}}}

\newcommand{\DeltaE}[0]{\Delta_{\mathbb{E}}}

\newcommand{\h}[0]{\mathbb{H}^{n}}
\newcommand{\hf}[0]{\mathbb{H}^{1}}
\newcommand{\g}[0]{\mathbb{G}}
\newcommand{\e}[0]{\mathbb{E}}
\newcommand{\ps}[2]{\left<#1,#2\right>}
\newcommand{\BG}[2]{B_{#1}^{\g}(#2)}

\newcommand{\dl}[0]{\delta_{\lambda}}
\newcommand{\dr}[0]{\delta_{r}}
\let\P\undefined 
\newcommand{\P}[0]{\mathcal{P}}
\newcommand{\gradG}[0]{\nabla_{\mathbb{G}}}
\newcommand{\gradE}[0]{\nabla_{\mathbb{E}}}

\newcommand{\nor}[1]{\left\lvert#1\right\rvert}
\newtheorem{theorem}{Theorem}
\numberwithin{theorem}{section}
\newtheorem{prop}[theorem]{Proposition}

\theoremstyle{definition}
\theoremstyle{remark}
\newtheorem{nota}[theorem]{Remark}
\numberwithin{equation}{section}
 \DeclareMathOperator*{\Span}{span}

\begin{document}
\title[\dots Some counterexamples to Alt-Caffarelli-Friedman \dots in Carnot groups] {Some counterexamples to Alt-Caffarelli-Friedman monotonicity formulas in Carnot groups}
\author[]{Fausto Ferrari and Davide Giovagnoli}
\date{\today}
\address{Fausto Ferrari: Dipartimento di Matematica, Università di Bologna, Piazza di Porta \\ S.Donato 5, 40126, Bologna-Italy}
\email{fausto.ferrari@unibo.it}
\address{Davide Giovagnoli: Dipartimento di Matematica, Università di Bologna, Piazza di Porta \\ S.Donato 5, 40126, Bologna-Italy}
\email{d.giovagnoli@unibo.it}
\thanks{F.F. is partially supported by 2024-INDAM-GNAMPA project: {\it Free boundary problems in noncommutative structures and degenerate operators.}}
\maketitle

\let\thefootnote\relax

\tableofcontents 
\begin{abstract}
In this paper we continue the analysis of an Alt-Caffarelli-Friedman (ACF) monotonicity formula in Carnot groups of step $s >1$ confirming the existence of counterexamples to the monotone increasing behavior. In particular, we provide a sufficient condition that implies the existence of some counterexamples to the monotone increasing behavior of the ACF formula in Carnot groups. 
 The main tool is based on the lack of orthogonality of harmonic polynomials in Carnot groups. This paper generalizes the results proved in \cite{ferrari2023counterexample}.
\end{abstract} 

\bigskip

\section{Introduction}

The regularity of the solutions of two-phase free boundary problem finds crucial the use of monotonicity formulas, such as the one introduced in \cite{alt1984variational}. In this note, we continue to investigate this subject in the noncommutative setting of Carnot groups, improving the results already obtained in  \cite{ferrari2020some, ferrari2020new, ferrari2023counterexample, ferrari2023alt}. In fact, this paper follows the main stream of the research started in \cite{ferrari2020some} and its aim is to extend some nonexistence results obtained first in \cite{ferrari2023counterexample}, in the framework of the first Heisenberg group $\hf$, to a larger class of Carnot groups. 

 More precisely, in Theorem \ref{t1}, we state a sufficient condition under which
\begin{equation} \label{c1}\Phi(r)=\frac{1}{r^2} \int_{\BG{r}{0}} \nor{\gradG u(M)}^2 \GammaG(M) \, dM 
\end{equation}
satisfies a monotone decreasing behavior in a right neighborhood of $0$, 
for a fixed nice $\g$-harmonic function $u$ such that $u(0)=0$, where $\GammaG$ denotes the fundamental solution of the subLaplacian $\Delta_{\g}$ in a Carnot group $\g$ and $\nabla_{\g}u$ the horizontal intrinsic gradient of $u$. 

 The results have been obtained exploiting the non-commutative features of the Carnot groups and they exhibit a direct method to construct counterexamples to the simplest  Alt-Caffarelli-Friedman formula \eqref{c1} in all the Carnot groups.

In order to get a better overview  about the relevance of this result, we recall the role of such monotonicity formulas in two-phase free boundary problems.

To our knowledge,
the first formalization of the monotonicity formula applied to a two-phase free boundary problem appeared in \cite{alt1984variational}, considering functions  that are minimizers of the energy functional
\begin{equation}\label{Bernoulli}
\mathcal{J}(u) = \int_{\Omega} \nor{\nabla u}^{2} + \lambda_{+} \chi_{ \{u>0 \}} + \lambda_{-} \chi_{\{u\leq 0 \}} \, dx,
\end{equation}
 where $\Omega \subset \mathbb{R}^n$ is an open set endowed with locally Lipschitz boundary and $\lambda_{\pm} >0$ some given positive numbers.
  
  Minima of $\mathcal{J}$ belong to some suitable subsets of the Sobolev space $H^1(\Omega)$, determined assuming particular conditions on $u$, see \cite{alt1984variational}, on which we don't wish to enter here. 
 
 Nevertheless, in this framework, a special role is played by the set $F(u):=\partial\Omega^{+}(u) \cap \Omega$, traditionally named as  {\it the free boundary} of the problem.  

Assuming some further hypotheses, see \cite{alt1984variational} one more time,  the condition that minima satisfy on $F(u)$ may be understood as 
$$
\nor{ \nabla u^+}^2 - \nor{\nabla u^-}^2 = \lambda_{+} - \lambda_{-}.
$$
In addition, minima of $\mathcal{J}$, in the sense of the {\it domain variation}, \cite{Weiss98}, satisfy  the following system studied in \cite{caffarelli1988harnack},
\begin{equation} \label{i1}
 \left\{  \begin{array}{lclclc}
       \Delta u=0&  \qquad \text{in $\Omega^{+}(u):=\{ x \in \Omega : u > 0 \}$},\\
       \Delta u=0&  \qquad \text{in $\Omega^{-}(u):=\{ x \in \Omega : u \leq 0 \}^{o}$},\\
       \nor{ \nabla u^+}^2 - \nor{\nabla u^-}^2 = \lambda_{+} - \lambda_{-}  & \qquad \text{on $ F(u)$,}  
 \end{array}\right.\end{equation}
with viscosity tools, see also \cite{dzhugan2020domain} for a gentle introduction about the domain variation approach. 

Problem \eqref{i1} may be
  read as entailing the Euler-Lagrange equations associate with the minima of the functional \eqref{Bernoulli}.

In this setting, supposing that $0 \in F(u)$,  the monotonicity result, proved in \cite{alt1984variational}, states that for every 
 solution $u \in H^1(\Omega)$  of \eqref{i1}, denoting $ u^+:= \sup \{u,0 \}$ and $u^-:=\sup \{ -u,0 \}$ the positive and negative part of $u$  respectively, the function 
\begin{equation} \label{i2}
    J_{u}(r)= \frac{1}{r^4} \int_{B_{r}(0)} \frac{\nor{\nabla u^+}^2}{\nor{x}^{n-2}}\, dx  \int_{B_{r}(0)} \frac{\nor{\nabla u^-}^2}{\nor{x}^{n-2}}\, dx
\end{equation}
is monotone increasing for all $r \in (0,R),$  for a suitable  $R>0$. 

 Function \eqref{i2}, and its monotonicity properties, has been widely studied in the Euclidean setting and, usually, it is called the Alt-Caffarelli-Friedman monotonicity formula (ACF formula or, simply, monotonicity formula for future references) since, in this framework, it is monotone increasing in a right neighborhood of $0$. 
 
 In addition,  in the same paper \cite{alt1984variational}, it has been shown how \eqref{i2} basically provides bounds to the product of the gradients of $u^+ $ and $u^-$ at the points of the free boundary. Thus, this monotonicity formula allows to deduce the Lipschitz continuity of the global solution $u$ of the problem \eqref{i1}.

  ACF formula has several applications in many different framework, see for instance \cite{caffarelli2005geometric,silva2019recent, allen2021sharp, tortone2022liouville,de2021rectifiability,maiale2021boundary} and  it has had several generalizations like in \cite{caffarelli1988harnack,caffarellimonotonicity,caffarelli1998gradient, caffarelli2002some,conti2005asymptotic,edquist2008parabolic,teixeira2011monotonicity, matevosyan2011almost, velichkov2014note,soave2022anisotropic,garofalo2023note}. 
 
 Hence,  the same questions posed in the Euclidean setting,  concerning the regularity of functions satisfying companion systems like \eqref{i1}, arise when the problem is stated in Carnot groups,  where, as it is well known, it is possibile to study the parallel problem of \eqref{i1} that appears to have the following form, see \cite{dzhugan2020domain},
\begin{equation} \label{c5}
 \left\{  \begin{array}{lclc}
       \Delta_{\g} u=0&  \qquad \text{in $\Omega^{+}(u)$},\\
       \Delta_{\g} u=0&  \qquad \text{in $\Omega^{-}(u)$}, \\
       \lvert \nabla_{\g} u^{+} \rvert^{2} - \lvert \nabla_{\g} u^{-} \rvert^{2}  = 1 & \qquad \text{on $ F(u).$}  
  \end{array}\right.\end{equation}
   In particular, here, $\DeltaG$ and $\gradG$ denote respectively the sub-Laplacian and the horizontal gradient on $\g$, we refer to Section \ref{section2} for precise definitions. 

 As a consequence, to investigate the existence of an {\it intrinsic} companion ACF formula of \eqref{i2}, in the noncommutative framework,  appears, in a sense, natural and useful as well.

The Free Boundary Problem \eqref{c5} suggests that the ACF formula candidate to this two-phase scenario  has to be the following one:
\begin{equation} \label{c4}
 J_{u}^{\mathbb{G}}(r) =\frac{1}{r^4} \int_{B_{r}^{\mathbb{G}}(0)} \lvert \nabla_{\mathbb{G}} u^+(M) \rvert^2 \Gamma_{\mathbb{G}}(M) \, dM \cdot \int_{B_{r}^{\mathbb{G}}(0)} \lvert \nabla_{\mathbb{G}} u^-(M) \rvert^2 \Gamma_{\mathbb{G}}(M) \, dM,\end{equation}
 where $\GammaG$ denotes the fundamental solution of $\Delta_{\g} $, with pole at the origin, and $B_{r}^{\mathbb{G}}(0)$ is the right superlevel set of $\GammaG$.
 
 It is worth to warn here the reader about some delicate points. Instead of considering, in the integrals of \eqref{c1} and \eqref{c4}, the fundamental solution of the sub-Laplacian, some powers of different homogeneous norms might be considered. For instance, instead of the so called $\mathcal{L}-$gauge norm associated with the fundamental solution, it is possible to deal with 
 $d_{CC}^{2-Q}$, where $d_{CC}$   denotes the Carnot-Charath\'eodory distance in the Carnot group and $Q$ is the homogeneous dimension. In addition, instead of considering the horizontal gradient $\nabla_{\mathbb{G}}$ it is possible to study the right horizontal gradients, see for instance \cite{garofalo2023note}. 
 
 Anyhow, in view of the problem \eqref{c5} and  recalling some of the applications associated with the regularity of its solutions as well, we are interested in the properties of the functions defined in \eqref{c1} and \eqref{c4}.

In this direction some results have been already obtained: in \cite{ferrari2020new,ferrari2020some}
the authors discuss about the form of \eqref{c4}  and retrace a part of the fundamental tools in parallel with the Euclidean proofs for the Heisenberg group $\h$. Following an idea recalled in \cite{petrosyan2012regularity}, about classical harmonic functions, through the use of harmonic homogeneous polynomials, in \cite{ferrari2023counterexample} it has been proved an explicit counterexample in $\hf$  to the monotone increasing behavior of the ACF functionals \eqref{c1} and \eqref{c4}. 

More recently, in \cite{ferrari2023alt}, a sufficient condition involving a mean value formula of the norm of the gradient has been established and used to provide another family of counterexamples in $\hf$. 

 We point out as well that, in \cite{garofalo2023note}, the author studied a monotonicity formula with right invariant vector fields in a different perspective of research.


The main result we obtain in this paper is the following one.
\begin{theorem}\label{t2}
    For any Carnot group $\g$ of step $s$, with $s >1$, there exists an intrinsic harmonic function $u$ such that \eqref{c1}  fails to be monotone increasing in a right neighborhood of $0$.
\end{theorem}

 Indeed, in Theorem \ref{t2}, we prove that in every Carnot group an intrinsic harmonic function exists such that
\eqref{c1} is monotone decreasing.

On the other hand, since \eqref{c1} appears also to be 
a factor of \eqref{c4}, this fact may be exploited to prove, as well as in \cite{ferrari2023counterexample} in the Heisenberg group only, that Theorem \ref{t3}, holds. 

More precisely, we can generalize the result of the nonexistence of an ACF formula as \eqref{c4} to all Carnot groups of step two. In fact the following result holds.
\begin{theorem} \label{t0}
     For any Carnot group $\g$ of step $2$, there exists a continuous function $u$ such that $u$ is harmonic in $\{u>0\}$ as well as $u$ is harmonic in $\{u\leq 0\}^o$ and $J_{u}^{\g}$  fails to be increasing in a right neighborhood of $0$.
\end{theorem}
We proved this result in such generality only in two step Carnot groups since we can exploit both a symmetry property with respect to the variables associated to the first stratum of the function in the counterexample built starting from  Theorem \ref{t2}, and the symmetry property of the  fundamental solution with respect the first stratum variables, as it can be recognized in the representation formula due to Beals, Gaveau and Greiner \cite{beals1996green}.
 
 Carnot group of step $s >2$, in general, to our knowledge don't have a so explicit representation of the fundamental solution, even if, it is well known that $\Gamma_{\g} (P)=\Gamma_\g(P^{-1}),$ see for instance \cite{Folland_main}, where $\Gamma_{\g}$ denotes the fundamental solution, with pole in $0$, of the sub-Laplacian $\Delta_{\mathbb{G}}$ in the Carnot group $\mathbb{G}$.
 
 Hence, Theorem \ref{t3} applies only whenever it is possible to build a counterexample with an \textit{intrinsic odd} property, namely $u(P)=-u(P^{-1})$, see Section \ref{AppEngel}.


We point out that Theorem \ref{t2} is a consequence of Theorem \ref{t1}, where we state a sufficient condition that implies the decreasing behavior of \eqref{c1} in all the Carnot groups. 
 Section \ref{section2} provides a self contained notation used in this paper, Section \ref{Section3} describes the main tools that we exploit to obtain the counterexamples, in particular the proof of Theorem \ref{t1}. Section \ref{proof} is devoted to the proof of the main result. Section \ref{non} gives some sufficient conditions for the nonincreasing behavior of \eqref{c4} and contains the proof of Theorem \ref{t0}. Eventually, Section \ref{AppEngel} focuses on an explicit counterexample to the monotone increasing behavior of  \eqref{c1} in the three step Engel group.

\section{Carnot group settings} \label{section2}
In this section we introduce the main definitions and notations used in the framework of Carnot groups.
For the notations we follow \cite{ferrari2015harnack} and we refer to \cite{bonfiglioli2007stratified},  and \cite{vit08} for a broader exposition  of Carnot groups.

    A Lie group $\mathbb{G}$ is a manifold endowed with a structure of differential group, i.e. a group where the maps
    \[ (x,y) \mapsto xy \in \mathbb{G}, \quad x \mapsto x^{-1} \in \mathbb{G} \text{ are $C^{\infty}$ for any $x,y \in \mathbb{G}$}. \]
    A vector space $\mathfrak{g}$ is said to be a Lie algebra if there exists a bilinear and anti-symmetric map $[ \cdot , \cdot]:\mathfrak{g} \times \mathfrak{g} \mapsto \mathfrak{g}$ which satisfies the Jacobi's identity, i.e.,
    \[ [X, [Y,Z]] + [Y, [Z,X]] + [Z, [X,Y]]= 0 \quad \text{for all $X, Y, Z \in \mathfrak{g}.$}\]

Given two subalgebras $\mathfrak{a}, \mathfrak{b}$ of a Lie algebra $\mathfrak{g}$ we will denote with $[\mathfrak{a}, \mathfrak{b}]$ the vector subspace generated by the elements of $\{ [X,Y] \, : \, X \in \mathfrak{a} , Y \in \mathfrak{b} \}.$ We denote $\mathfrak{g}^{1}:= \mathfrak{g}$ and by induction $\mathfrak{g}^{l+1}:= [\mathfrak{g}, \mathfrak{g}^{l}]$. \\
We will say that $\mathfrak{g}$ is nilpotent of step $s$ if $\mathfrak{g}^{s} \neq \{ 0 \}$ and  $\mathfrak{g}^{s+1} = \{ 0 \}.$  \\

    A connected and simply connected stratified nilpotent Lie group $\mathbb{G}$ is said to be a Carnot group of step $s$ if its associated Lie algebra $\mathfrak{g}$ admits a stratification of step $s$,  if  $\mathfrak{g}_1, \dots, \mathfrak{g}_s$ linear subspaces exist such that, $\mathfrak{g}$ can be written as the direct sum of the $\mathfrak{g}_i$ and the $(i+1)$-th subspace is generated by commutating the elements of $\mathfrak{g}_1$ and $\mathfrak{g}_i$, i.e.
\[ \mathfrak{g}=\mathfrak{g}_1 \oplus \dots \oplus \mathfrak{g}_p, \quad [\mathfrak{g}_1, \mathfrak{g}_i]=\mathfrak{g}_{i+1}, \quad \mathfrak{g}_s \neq \{ 0\}, \quad \mathfrak{g}_i=\{ 0\} \text{ for $i > s$}.\]

The first layer $\mathfrak{g}_1$, often called horizontal layer, has a key role since can generate the whole space $\mathfrak{g}$ by commutation.  \\
The homogeneous dimension $Q$ of $\mathbb{G}$ is,
\[ Q:=\sum_{i=1}^{p} i \, dim(\mathfrak{g}_i).\]
Let us denote with $e$ the unit element on $\mathbb{G}$.   We recall that the map  $X \mapsto X(e)$, that associates with a left-invariant vector field $X$ its value at $e$, is an isomorphism from $\mathfrak{g}$ to the tangent space $T\mathbb{G}_{e}$, identified with $\mathbb{R}^{n}$. \\
 Let $m_i= dim(\mathfrak{g}_i) $ and $h_i=m_1+\dots+m_i$ for $i=1, \dots, p$.   Hence it holds $h_p=n$. 
We choose the basis $e_1, \dots, e_{n}$ of $\mathbb{R}^{n}$ adapted to the stratification of $\mathfrak{g}$, in the sense of having
\[ e_{h_{j-1}+1}, \dots, e_{h_{j}} \text{ as the basis of } \mathfrak{g}_j \text{ for each }j=1,\dots, p.\]
Moreover, let $\{X_1, \dots, X_{n} \}$ be the family of left invariant vector fields such that $X_{i}(e)=e_{i}$, $i= 1, \dots, n.$ 
The subbundle of the tangent bundle $T\g$  spanned by the vector fields associated to the horizontal layer $X_1, \dots, X_{m_1}$ is called the horizontal bundle $H\g$. 

For every $x \in \g$, each fiber $H\g_{x}=span \{ X_1(x), \dots, X_{m_1}(x)\}$  is endowed by a scalar product $\ps{\cdot}{\cdot}$ such that $\{X_1(x), \dots, X_{m_1}(x)\}$ becomes an orthonormal basis.

We identify the Carnot group $\g$, through exponential coordinates, with the Euclidean space $(\mathbb{R}^{n}, \cdot)$, where $n$ is the dimension of $\mathfrak{g}$, endowed with a suitable group operation. \\
For any $x \in \g$, the left translation $\tau_x: \g \mapsto \g$ is defined as \[ z \mapsto \tau_x z := x \cdot z.\]
For any $\lambda >0$, the dilations $\dl:\g \mapsto \g,$ is defined as 
\[ \dl(x_1, \dots, x_n)= (\lambda^{d_1}x_1, \dots, \lambda^{d_n}x_n),\]
where $d_i \in \mathbb{N}$ is called homogeneity of the variable $x_i \in \g$ (see \cite{folland1982hardy} Chapter 1) and is defined as 
\[ d_j=i \quad \text{whenever } h_{i-1}+1 \leq j \leq h_i,\]
hence, $1=d_1=\dots=d_{m_1} < 2=d_{m_1+1} \leq \dots \leq d_n=p.$ \\
We follow the notation of \cite{bonfiglioli2007stratified}, Chapter 20,  so that $\beta$ denotes  a multi-index with $n$ entries, $\beta=(\beta_1, \dots, \beta_n) \in \mathbb{N}^{n}$.  Let
\begin{equation*}
\begin{aligned} x^{\beta}&=x_1^{\beta_1}\cdot \cdot \cdot x_n^{\beta_n}, \\ \nor{\beta}=\beta_1 + \dots + \beta_n, &\quad \nor{\beta}_{\g}=d_1\beta_1+ \dots + d_n \beta_n.\end{aligned} \end{equation*}
We introduce  the $\g$-polynomial as the polynomial  $\P$ with respect to the coordinate system $(x_1, \dots, x_n) \in \g$ and we define the $\g$-degree of a $\g$-polynomial as follows,
\begin{equation} \label{degG}
    \deg_{\g}(\P) := \max \left\{ \nor{\beta}_{\g} \; : \; \P(x)=\sum_{\beta \in \mathbb{N}^{n}} c_{\beta} x^{\beta} \text{ with $c_{\beta} \neq 0$ for every $\beta$}  \right\}.
\end{equation}
From now on, for simplicity, a  $\g$-polynomial and the $\g$-degree of a   $\g$-polynomial are denoted omitting the letter $\g$. \\
We exploit the Haar measure of $\g$ which is the Lebesgue measure in $\mathbb{R}^{n}$ to have a notion of an integral in $\g$. Once the basis $X_1, \dots, X_{m_1}$ of the horizontal layer is fixed, we define for any function $f: \g \mapsto \mathbb{R},$ for which  $X_{j}f$  exists, the horizontal gradient of $f$, denoted by $\gradG f$, as the horizontal section
\[ \gradG f= \sum_{i=1}^{m_1} (X_{i}f)X_{i},\]
whose coordinates are $(X_1f, \dots, X_{m_1}f)$. In the same way it is possible to extend to higher order  this rappresentation.
Moreover, if $\phi= (\phi_1 \dots, \phi_{m_1})$ is  a horizontal section such that $X_j \phi_j \in L_{loc}^{1}(\g)$ for $j=1,\dots, m_1,$ we define $\divG \phi$ as the real valued function
\[ \divG ( \phi) := \sum_{j=1}^{m_1} X_{j} \phi_j.\]
We denote by $\DeltaG$ the sub-Laplacian on $\g$ the operator
\[ \DeltaG = \sum_{j=1}^{m_1} X_j^{2}.\]
As well as to the Euclidean case, we are interested in fundamental solution of $\DeltaG$.
For a deeper overview on the existence, the properties and the global estimates of fundamental solutions of sub-Laplacians we refer, for instance, to  \cite{folland1982hardy,bonfiglioli2007stratified}. 
 In particular, a function $\GammaG(\cdot, P):\g\setminus \{ P\} \mapsto \mathbb{R}$ is the fundamental solution of $\DeltaG$ with pole in $P \in \g$ if:
\begin{itemize}
    \item[(i)] $\GammaG(\cdot, P) \in C^{\infty}(\g \setminus \{ P\})$;
    \item[(ii)] $\GammaG(\cdot, P) \in L_{loc}^{1}(\g)$ and $\GammaG(M,P) \rightarrow 0$ when $M$ tends to infinity;
    \item[(iii)] $\DeltaG \GammaG(\cdot, P) = \delta_{P}$, being $\delta_{P}$ the Dirac measure supported at ${P}$. More explicitly,
    \[ \int_{\g} \GammaG(M,P) \DeltaG \varphi(M) dM = -\varphi(P) \quad \text{for any $\varphi \in C_{0}^{\infty}(\g). $} \]
\end{itemize}
We further introduce, as well, the $\g$-ball centered at $P$ with radius $r$, $\BG{r}{P}$, as the superlevel set of the fundamental solution $\GammaG(\cdot,P)$, \cite{bonfiglioli2007stratified},
 $$\BG{r}{P}= \{ M \in \g \, : \, \GammaG(M, P) \geq r^{-(Q-2)}   \}.$$

\section{A characterization via harmonic homogeneous polynomials}\label{Section3}
In this framework, each of two  factors of the ACF Formula in Carnot groups has the form,
\begin{equation} 
 \Phi(r) =\frac{1}{r^2} \int_{\BG{r}{P}} \nor{\gradG u(M)}^2 \GammaG(M,P) \, dM.\end{equation}

A key role in establishing an ACF Monotonicity Formula relies on the properties of the fundamental solution. 

However,  an explicit formula of the fundamental solution can be given only for a small class of Carnot groups. For instance, apart the Euclidean case, it can be done explicitly dealing with $H-$type groups that include Heisenberg groups, see \cite{folland1973fundamental}. 

In general the fundamental solution of a sub-Laplacian in a Carnot group is not known explicitly. Nevertheless, see \cite{bonfiglioli2007stratified}, a homogeneous norm exists $d(\cdot):= \nor{\cdot}$ such that $\GammaG(M,P)= C_{Q}d^{2-Q}(M^{-1}\circ P)$. Then it is possible to exploit the scale invariant properties of the fundamental solutions as well as those of the intrinsic harmonic functions to get a useful representation of \eqref{c1}. \newpage

For any $\lambda \in \mathbb{R}$, let $u_{\lambda}(P):=u(\dl P)$.  Then, the following properties hold:
\begin{itemize}
    \item[(G1)]  $\gradG u_{\lambda}(P)= \lambda (\gradG u)(\dl P);$
    \item[(G2)] If $u$ satisfies $\DeltaG u=0$ then $\DeltaG u_{\lambda}=0$;
    \item[(G3)] For any $M \in \g,$ $\GammaG(\dl M,P)=\lambda^{2-Q} \GammaG(M,P)$. 
\end{itemize}

The main result of this section is the following one.
\begin{theorem} \label{t1}
Let $\g$ be a Carnot group. If $\P_1$ and $\P_3$ are homogeneous harmonic polynomials of degree  $1$ and $3$ respectively, such that
\begin{equation}\label{c2}
     a_{2}: = \int_{\BG{1}{0}} \displaystyle \ps{\gradG \P_1(P)}{\gradG \P_3(P)}  \GammaG(P) \, dP \text{ is positive.} \end{equation}
Then, the function $u:=\P_1 - \P_3$ enjoys $\DeltaG u=0$ and \eqref{c1} is monotone decreasing in a right neighborhood of $0$.
\end{theorem}

\begin{proof} 
Let
    \begin{align*}
        \Phi(r)&=\frac{1}{r^2} \int_{\BG{r}{0}} \nor{\gradG u(M)}^2 \GammaG(M) \, dM \\
        &= \frac{1}{r^2} \int_{\BG{1}{0}} \nor{\gradG u(\dr P)}^2 \GammaG(\dr P) r^{Q}  \, dP 
        \intertext{where $M \in \BG{r}{0}$ is seen as $\dr P$, with $P \in \BG{1}{0},$}
        &\underbrace{=}_{(G1)} \frac{1}{r^2} \int_{\BG{1}{0}} \nor{\gradG (u(\dr P)) r^{-1}}^2 \GammaG(\dr P) r^{Q}  \, dP \\
        &\underbrace{=}_{(G3)}  \frac{1}{r^2} \int_{\BG{1}{0}} \nor{\gradG (u(\dr P)) r^{-1}}^2 r^{2-Q}\GammaG(P) r^{Q}  \, dP \\
        &= \int_{\BG{1}{0}} \nor{\gradG (u(\dr P)) r^{-1}}^2 \GammaG(P) \, dP. 
    \end{align*}
Suppose now that $u$ is the sum of two homogeneous harmonic polynomials. Let $\P_m$ and $\P_h$ be two homogeneous harmonic polynomials of degree $m$ and $h$ respectively, i.e. 
\[ \DeltaG(\alpha\P_m+ \beta \P_h)=0, \quad \alpha,\beta \in \mathbb{R},\]
by homogeneity one also have that $\P_m(\dr P)=\lambda^{m} \P_m(P)$ so,
\[ \lambda [\gradG \P_m](\dl P)\underbrace{=}_{(G1)}\gradG(\P_m(\dl P)) = \lambda^{m} [\gradG \P_m](P), \]
therefore  it holds
\[ [\gradG \P_m](\dl P)= \lambda^{m-1} [\gradG \P_m](P),\]
and an analogous result hold for  $\P_h$ as well.

Hence, $\nor{\gradG(\P_m + \P_h)(\dl P)}= \nor{\lambda^{m-1 }\gradG \P_m(P) + \lambda^{h-1} \gradG \P_h(P)}$. \\
Let us select a polynomial of degree $1$ and one of degree $3$, both intrinsically harmonic.  Denote
\[ u(P):=\P_1(P) - \P_3(P), \]
where the lower index denotes the homogeneity of the polynomial.  
As a consequence we get,
\begin{align*}
    \Phi(r)&= \int_{\BG{1}{0}} \displaystyle \nor{\gradG \biggl(\frac{\P_1(\dr P)}{r}\biggl)- \gradG \biggl(\frac{\P_3(\dr P)}{r}\biggl)}^2 \GammaG(P) \, dP \\
    &= \int_{\BG{1}{0}} \displaystyle \biggl( \nor{\gradG \biggl(\frac{\P_1(\dr P)}{r}\biggl)}^2 + \nor{\gradG \biggl(\frac{\P_3(\dr P)}{r}\biggl)}^2 \biggl) \GammaG(P) \, dP + \\
    & \, - 2\int_{\BG{1}{0}} \displaystyle \ps{\gradG \biggl(\frac{\P_1(\dr P)}{r}\biggl)}{\gradG \biggl(\frac{\P_3(\dr P)}{r}\biggl)} \GammaG(P) \, dP,
    \intertext{by homogeneity of $\P_1, \P_3$,} 
    &= \int_{\BG{1}{0}} (\nor{\gradG \P_1(P)}^2+r^4 \nor{\gradG \P_3(P)}^2) \GammaG(P) \, dP + \\
    & \, - 2 \int_{\BG{1}{0}} \displaystyle \ps{\gradG \P_1(P)}{\gradG \P_3(P)} r^2 \GammaG(P) \, dP,
    \intertext{so that we obtain,}
    &= a_{0} - 2r^2 a_2 + a_4 r^{4},\\
        \text{where:}  \: a_{0} &= \int_{\BG{1}{0}} \nor{\gradG \P_1(P)}^2 \GammaG(P) \, dP ;\\
        a_{2} &= \int_{\BG{1}{0}} \displaystyle \ps{\gradG \P_1(P)}{\gradG \P_3(P)}  \GammaG(P) \, dP ;  \\
        a_{4} &=   \int_{\BG{1}{0}} \nor{\gradG \P_3(P)}^2 \GammaG(P) \, dP.
    \end{align*}
The proof is complete, since $ \Phi'(r)=-4 a_2 \, r + 4 a_4\, r^3 $ is negative in a neighborhood of $0$. For instance if $r \leq \sqrt{\frac{a_2}{a_4}}$.
\end{proof}
\begin{nota}
    Theorem \ref{t1} cannot be applied to the trivial case of the Euclidean $\mathbb{R}^{n}$  since the Euclidean harmonic polynomial are orthogonal in the following sense.  For every $\P_k,\P_h $  harmonic homogeneous polynomials of degree $h$ and $k$ respectively
    \begin{equation}\label{euclidean}
        \int_{B_1(0)} \ps{\nabla \P_h(x)}{\nabla \P_k(x)} \nor{x}^{2-n} dx = 0, \qquad h\neq k.
    \end{equation}
\end{nota}
 We recall here the short proof, see e.g. \cite{Muller66}.
    From coarea formula we obtain
    \begin{equation}
         \int_{B_1(0)} \ps{\nabla \P_h(x)}{\nabla \P_k(x)} \nor{x}^{2-n} dx = \int_{0}^{1} \rho^{2-n} \int_{\partial B_{\rho}(0)} \ps{\nabla \P_h(x)}{\nabla \P_k(x)} \, d\sigma(x) \, d\rho.
    \end{equation}
Each component of the gradient of $\P_h$ and $\P_k$ are still homogeneous polynomials so we obtain the thesis if 
\begin{equation} \label{ortho}
    \int_{\partial B_{\rho}(0) } \P_h(x) \P_k(x) \, d\sigma(x) = 0.
\end{equation}
The equation \eqref{ortho} is a consequence of the divergence theorem. Indeed, being $\frac{x}{\rho}$ the  outward normal to $\partial B_{\rho}(0)$, 
\begin{equation}
    \begin{aligned} \label{ap2}
        0&= \int_{B_{\rho}(0)} (\P_h \Delta \P_k - \P_k \Delta P_h) \, dx = \int_{B_{\rho}(0)} \div (\P_h\nabla\P_k - \P_k \nabla \P_h) \, dx \\
        &= \int_{\partial B_{\rho}(0)} \ps{\P_h\nabla\P_k - \P_k \nabla \P_h}{\frac{x}{\rho}} \, d\sigma(x) \\ &= \int_{\partial B_{\rho}(0)} (\P_h \ps{\nabla\P_k}{\nu}- \P_k \ps{\nabla\P_h}{\nu}) \, d\sigma(x),
    \end{aligned}
    \end{equation}
   where $\nu:=\frac{x}{\rho}.$
    Let $\beta \in \mathbb{N}^{n} $ and
    $$ \P_k(x)=\sum_{\beta, \nor{\beta}=k} b_{\beta} x^{\beta},$$ then
    \begin{equation} \label{ap0}
    \nabla \P_k(x) = \sum_{\beta, \nor{\beta}=k} b_{\beta} (\beta_1 x_1^{\beta_1-1}x_2^{\beta_2}\dots x_n^{\beta_n}, \dots , \beta_n x_1^{\beta_1}\dots x_n^{\beta_n-1}).\end{equation}
    \begin{equation}\begin{aligned} \label{ap3}
        \ps{\nabla \P_k}{\nu}&= \frac{1}{\rho}\sum_{\beta, \nor{\beta}=k} b_{\beta} (\beta_1 x_1^{\beta_1}x_2^{\beta_2}\dots x_n^{\beta_n}+ \dots + \beta_n x_1^{\beta_1}\dots x_n^{\beta_n}) \\
        &= \frac{1}{\rho}\sum_{\beta, \nor{\beta}=k} b_{\beta} (\beta_1 + \dots + \beta_n) x_1^{\beta_1}x_2^{\beta_2}\dots x_n^{\beta_n}= \frac{1}{\rho}\sum_{\beta, \nor{\beta}=k} b_{\beta} \nor{\beta} x^{\beta} = \frac{k \P_k}{\rho}.    \end{aligned}
        \end{equation}
        Applying \eqref{ap3} to \eqref{ap2}, we get
        \begin{equation*}
            0 = \frac{1}{\rho}\int_{\partial B_{\rho}(0)}  (\P_h \ps{\nabla\P_k}{x}- \P_k \ps{\nabla\P_h}{x}) \, d\sigma(x) =\frac{1}{\rho} \int_{\partial B_{\rho}(0)} (\P_h k \P_k - \P_k h \P_h) \, d\sigma(x),
        \end{equation*}
        that gives for $h\neq k$
        $$\int_{\partial B_{\rho}(0)} \P_h \P_k \, d\sigma = 0.$$

\section{The general case} \label{proof}
In this section we investigate the condition on the positivity of $a_2$ in order to apply the same argument to all Carnot groups. We reduce the problem to  Carnot groups with step two and three.
Afterwards, we consider directly such groups and provide a method to build an explicit counterexample. 

First of all we recall a result about the representation of the vector fields belonging to the horizontal layer, which which is crucial to the reduction of the problem.
\begin{prop}[in \cite{folland1982hardy} is Proposition 1.26] \label{p1}
If $j=1,\dots, m_1$, the vector fields $X_j$ have polynomial coefficients and have the form
\begin{equation} \label{c3}
    X_j(x)= \partial_{x_j} + \sum_{k\geq 1; d_k>1} p_{j,k}(x) \, \partial_{x_k}
\end{equation}
where the $p_{j,k}$ are $\g$-homogeneous polynomials of degree $d_{k}-1$ for $d_k > 1.$
\end{prop}

As a consequence, we are in position to prove the main result, Theorem \ref{t2}

\begin{proof}(of Theorem \ref{t2}) Let us note that, by definition of $a_2$,
\[ \ps{\gradG \P_1(P)}{\gradG \P_3(P)} > 0\:\: \text{ a.e. in $\BG{1}{0}$} 
\]
 implies that
\[a_{2} := \int_{\BG{1}{0}} \displaystyle \ps{\gradG \P_1(P)}{\gradG \P_3(P)}  \GammaG(P) \, dP >0. \] 

Recalling the notation of Section \ref{section2} being $P \in \g$ as $P\equiv(x_1 ,\dots, x_n) \in \mathbb{R}^n$, a first degree homogeneous polynomial has to be  a linear combination of the first $m_1$ component,
\[ \P_1(x) = \sum_{k=1}^{m_1} b_{k}x_{k}.\]
Hence, from Proposition \ref{p1}, $\gradG \P_1(x)=(b_1, \dots, b_{m_1})$ is constant. \\
Note also that \eqref{c3} can be explicitly written as:
\begin{equation} \label{1}
 X_j= \partial_{x_j} + \sum_{k\geq 1; d_k=2} p_{j,k}^{1} \, \partial_{x_k}+ \sum_{k\geq 1; d_k=3} p_{j,k}^{2} \, \partial_{x_k} + T_{j} \end{equation}
with $$ T_j= \sum_{k\geq 1; d_k>3} p_{j,k}^{d_k -1} \, \partial_{x_k}.$$ 
A homogeneous harmonic polynomial of degree three, according to the definition of \eqref{degG}, is
\begin{equation*}
\P_3(x) = \sum_{\substack{\beta \in \mathbb{N}^{n}\\ \nor{\beta}_{\g}=3} } c_{\beta}x^{\beta}.
\end{equation*}
In particular, it cannot contain variables with homogeneity $d_k >3$. Thus,
if we apply $X_j$ to $\P_3,$ the vector field $T_j$ doesn't produce any contribution.
Arguing in this way, we handle the general case.

We start reducing ourselves to a simpler case, by considering a basic $2$-step Carnot group and exhibiting an explicit counterexample that may be generalized to groups with greater step.  

Thus, we have $\g$ whose Lie algebra is
     $\mathfrak{g}= \mathfrak{g}_1 \oplus \mathfrak{g}_2 $ with $\mathfrak{g}_1$ generated by $X_1, \dots, X_{m_1}$, $\mathfrak{g}_2$ by $Y_1,Y_2,\dots,Y_{n-m_1}$ and, via Proposition \ref{p1}, the horizontal vector fields are:
     \begin{equation}\label{2step}
         X_i=\partial_{x_i} + \sum_{j=1}^{n-m_1 } \left( \sum_{k=1}^{ m_1}\alpha_{kj}^{i}x_k \right) \partial_{y_j} \qquad i=1,\dots, m_1.
     \end{equation}
     We notice that some vector fields might commute, but being $\g$  a Carnot group of step two, it holds that for some $\Tilde{i},\Tilde{s} \in \{1, \dots,m_1 \}$ we have $[X_{\Tilde{i}},X_{\Tilde{s}}] \neq 0$ that together with \eqref{2step} means that exists at least a $\Tilde{j} \in \{1, \dots, n-m_1 \}$ for which
$\alpha_{\Tilde{s}\Tilde{j}}^{\Tilde{i}} \neq \alpha_{\Tilde{i}\Tilde{j}}^{\Tilde{s}}$,
in fact
 \begin{align*}
    [X_i, X_s] &= \biggl(\partial_{x_i} + \sum_{j=1}^{n-m_1 } \left( \sum_{k=1}^{ m_1}\alpha_{kj}^{i}x_k \right) \partial_{y_j}\biggl ) \biggl(\partial_{x_s} + \sum_{j=1}^{n-m_1 } \left( \sum_{k=1}^{ m_1}\alpha_{kj}^{s}x_k \right) \partial_{y_j}\biggl) \\
    & - \biggl(\partial_{x_s} + \sum_{j=1}^{n-m_1} \left( \sum_{k=1}^{ m_1}\alpha_{kj}^{s}x_k \right) \partial_{y_j}\biggl)  \biggl(\partial_{x_i} + \sum_{j=1}^{n-m_1} \left( \sum_{k=1}^{m_1}\alpha_{kj}^{i}x_k \right) \partial_{y_j}\biggl )  \\
    &= \sum_{j=1}^{n-m_1} (\alpha_{ij}^{s}- \alpha_{sj}^{i}) \partial_{y_j}.
\end{align*}
Let us simplify the notation considering $\Tilde{i}=1,\Tilde{s}=2$ and $y_{\Tilde{j}}=y$ with the corresponding vector fields $X_1,X_2$ and $Y=\partial_{y}$. Moreover denote with $(\alpha_{k}^{i})_{i,k=1,\dots,m_1}$ the coefficients with respect to $\partial_{y}$ via \eqref{2step}, so we have
\begin{equation} \label{cond1}
     [X_{\Tilde{i} }, X_{\Tilde{s}}]= [X_1, X_2]= (\alpha_{1}^{2}-\alpha_{2}^{1})\partial_{y} + \sum_{j=2}^{n-m_1} (\alpha_{1j}^{2}- \alpha_{2j}^{1}) \partial_{y_j}  \neq 0 \quad \text{and }\alpha_{1}^{2}-\alpha_{2}^{1} \neq 0 .
\end{equation}
Now, consider the polynomials  
\begin{equation}\begin{aligned}\P_1&=bx_2, \\ \label{choiceP3}\P_3 &= c_1 x_1^3+c_2x_1^2 x_2+ c_3x_1x_2^2+ c_4x_2^3 + c_5x_1y - c_5x_1 \left[ \sum_{i=3}^{m_1} \left( \frac{1}{2}  \alpha_{i}^{1}x_1x_i + \alpha_{i}^{2} x_2 x_i \right) \right], \end{aligned} \end{equation}
with $b,c_1,c_2,c_3,c_4,c_5 \in \mathbb{R}$ to be determined. 
For $\P_1$ we have
$$ \nabla_{\g} \P_1 =(0,b, 0, \dots, 0), \quad \DeltaG \P_1=0;$$
For $\P_3$, computing the derivatives, we get
\begin{align*}
    X_1 \P_3 &= (3c_1 + \alpha_1^{1} c_5)x_1^2 + (2c_2 + \alpha_2^{1} c_5)x_1x_2 +  c_3  x_2^2  +c_5 y \\&+ c_5x_1\left( \sum_{i=3}^{m_1}\alpha_{i}^{1}x_i \right)
     - c_5x_1\left( \sum_{i=3}^{m_1}\alpha_{i}^{1}x_i \right) - c_5 \left( \sum_{i=3}^{m_1}\alpha_{i}^{1}x_ix_2 \right);   \\ 
    X_2 \P_3&= (c_2 + \alpha_1^2 c_5)x_1^2 + (2c_3+\alpha_2^2 c_5)x_1x_2 + 3c_4 x_2^2 + c_5 x_1 \left( \sum_{i=3}^{m_1}\alpha_{i}^{2}x_i\right)  - c_5 \left( \sum_{i=3}^{m_1}\alpha_{i}^{2}x_ix_1 \right); \\
    \DeltaG \P_3 &= [ 6c_1 + 2c_3 + (\alpha_{2}^{2}+3\alpha_1^1) c_5]x_1 + [2c_2 + 6c_4 + 2\alpha_2^1 c_5]x_2 ;  \\
    \ps{\gradG \P_1}{\gradG \P_3} &=   b(c_2 + \alpha_{1}^{2} c_5)x_1^2  + b(2c_3+\alpha_{2}^{2} c_5)x_1x_2     + 3 b c_4x_2^2.
\end{align*}
Now, we impose the $\g$-harmonicity of $\P_1 - \P_3,$  furthermore we require  the cancellation of the mixed term of $\ps{\gradG \P_1}{\gradG \P_3}$ and we add two conditions to fix the positive sign. In this manner we obtain the following system:  
\begin{equation}\label{s2}
     \left\{  \begin{array}{llllll}
6c_1 + 2c_3 + (\alpha_{2}^{2}+3\alpha_1^1) c_5 = 0 \\
2c_2 + 6c_4 + 2\alpha_2^2 c_5  = 0 \\
 2b c_3 +  b \alpha_{2}^{2}c_5= 0 \\
bc_2 + b \alpha_{1}^{2} c_5 = p \\
 3b c_4 =q, 
 \end{array}\right.
 \end{equation}
 assuming the condition $(p,q) \in D,$  with $D := \{(x,y) \in \mathbb{R}^{2} \setminus\{ (0,0)\}, x \geq 0, y \geq 0 \}$.
 
 Fixed $b \in \mathbb{R}^{*}=\mathbb{R}\setminus\{0\}$, \eqref{s2} is a $5\times5$ linear system of the form $A_{b} \vec{c} = \Vec{v},$ where $\Vec{c}=(c_1, c_2, c_3,c_4, c_5)$, $\Vec{v}=(0,0,0,p,q)$ and
 \begin{equation} \label{m1}
     A_{b}=  \begin{pmatrix}
         6 & 0 & 2 & 0 & \alpha_{2}^{2} +3\alpha_{1}^{1}\\
         0 & 2 & 0 & 6 & 2\alpha_{2}^{1}\\
         0 & 0 & 2b & 0 & \alpha_{2}^{2} b \\
         0 & b & 0 & 0 & \alpha_{1}^{2} b  \\
         0 & 0 & 0& 3b & 0 
      \end{pmatrix} 
 \end{equation}
 For every $(\alpha_1^1,\alpha_2^1,\alpha_{1}^{2}, \alpha_{2}^{2}) \ \in \mathbb{R}^{4*}$ and $\alpha_2^1 \neq \alpha_{1}^{2}$ the system \eqref{s2} has a unique solution, since it is
 equivalent to the one associated with the following matrix obtained by the Gauss method:
  \begin{equation} \label{m2}
     \tilde{A}_{b}\vert \tilde{v}=  \begin{pmatrix}
         6 & 0 & 2 & 0 & \alpha_{2}^{2} +3\alpha_{1}^{1}\\
         0 & 2b & 0 & 6b & 2\alpha_{2}^{1}b\\
         0 & 0 & 2b & 0 & \alpha_{2}^{2} b \\
         0 & 0 & 0 & -6b & 2\alpha_{1}^{2} b-2\alpha_2^1b  \\
         0 & 0 & 0& 0 &  2\alpha_{1}^{2} b-2\alpha_2^1b 
      \end{pmatrix} \vert \begin{pmatrix}
          0\\
          0\\
          0\\
          p  \\
          p+q
      \end{pmatrix}. 
 \end{equation}
and since the determinant of $\tilde{A}_{b}$ is,
 $$ \det \tilde{A}_{b} = -72b^4(\alpha_{1}^{2} - \alpha_2^1).$$
Computing explicitly 
the resolving coefficients $c_i$ for $i=1, \dots, 5$ we have
\begin{equation} \label{coeff}
\begin{aligned}
    c_1 &= \frac{\alpha_1^1(p+q)}{2b(\alpha_{2}^{1}-\alpha_1^2)} , \\  
    c_2 &= \frac{\alpha_1^2 q +\alpha_{2}^{1} p}{b(\alpha_{2}^{1}-\alpha_1^2)},  \\ 
    c_3 &=  \frac{\alpha_{2}^{2}(p+q)}{2b(\alpha_{2}^{1}-\alpha_1^2)}, \\ 
    c_4 &= \frac{q}{3b}, \\  
    c_5 &=  \frac{p+q}{b(\alpha_{1}^{2} - \alpha_2^1)}.
 \end{aligned}
 \end{equation}
Hence, it is possible with this method to produce an explicit counterexample for a general Carnot group of step two. \\
The same procedure, with just technical complications, can berepeated considering a $3$-step Carnot group. Here, we suppose $\mathbb{R}^{n}$ as
\begin{equation*}
    \begin{aligned}
        \mathbb{R}^{n}&= \mathbb{R}^{m_1} \times \mathbb{R}^{m_2} \times \mathbb{R}^{n-m_1-m_2} \\
        x \in \mathbb{R}^{n}, \; x&=(x_1,x_2, \dots, x_{m_1}, y_{1}, \dots, y_{m_2},t_1, \dots, t_{n-m_1-m_2}).
    \end{aligned}
\end{equation*} Analogously to the previous argument, the corresponding basis of $\mathfrak{g},$  can be represented in such a way that
$X_1,X_2,\dots, X_{m_1}$ is a basis of $\mathfrak{g}_1$, $Y_1,\dots,Y_{m_2}$ is a basis of $\mathfrak{g}_2$, and $T_1, \dots, T_{n-m_1-m_2}$ as a basis of $\mathfrak{g}_3$. 
Without loss of generality, like in the previous case, we assume that
$X_1,X_2$ are such that $[X_1,X_2] \neq 0$,where:
 
\begin{equation}
    \begin{aligned} \label{3step}
        X_1&=\partial_{x_1} + \sum_{j=1}^{m_2} \left( \sum_{i=1}^{m_1} \gamma_{ij}^1 x_i\right) \partial_{y_j}+ \sum_{j=1}^{n-m_1-m_2} \left( \sum_{k,i=1}^{m_1} \delta_{kij}^{1}x_ix_k + \sum_{i=1}^{m_2} \theta_{ij}^{1}y_i\right) \partial_{t_{j}} \\
        &=\partial_{x_1} + \sum_{j=1}^{m_2} \left( \sum_{i=1}^{m_1} \gamma_{ij}^1 x_i\right) \partial_{y_j}+  \sum_{j=1}^{n-m_1-m_2} p_{1,j}^{2}  \partial_{t_{j}} \\
         X_2&=\partial_{x_2} + \sum_{j=1}^{m_2} \left( \sum_{i=1}^{m_1} \gamma_{ij}^2 x_i\right) \partial_{y_j}+ \sum_{j=1}^{n-m_1-m_2} \left( \sum_{k,i=1}^{m_1} \delta_{kij}^{2}x_ix_k + \sum_{i=1}^{m_2} \theta_{ij}^{2}y_i\right) \partial_{t_{j}} \\
         &=\partial_{x_2} + \sum_{j=1}^{m_2} \left( \sum_{i=1}^{m_1} \gamma_{ij}^2 x_i\right) \partial_{y_j} + \sum_{j=1}^{n-m_1-m_2} p_{2,j}^{2} \partial_{t_{j}}
    \end{aligned}
\end{equation}
The step two case allows to consider in \eqref{3step} the polynomials $p_{1,j}^{2}$ and $p_{2,j}^{2}$ to be nonzero. \\
Being
\begin{align*}
    [X_1,X_2]= \sum_{j=1}^{m_2} (\gamma_{1j}^{2} -\gamma_{2j}^{1})\partial_{y_{j}} + \sum_{j=1}^{n-m_1-m_2} \Tilde{p}_{1j} \partial_{t_j},
\end{align*}
where $\Tilde{p}_{1j}$ are homogeneous polynomials of degree $1$, but since the vector field has to belong to the second stratum, $\Tilde{p}_{1j}$ doesn't appear.

In order to have a step three stratified Lie group, if $X_1,X_2$ are defined as \eqref{3step}, then we have to require:
\begin{equation}
    \begin{aligned} \label{hp3}
       [X_1, X_2]&= (\gamma_{1}^{2} - \gamma_{2}^{1}) \partial_{y} + \sum_{j=2}^{m_2} (\gamma_{1j}^{2} -\gamma_{2j}^{1}) \partial_{y_{j}}  \neq 0, \\ 
        [X_1 , Y] &= k_1 \partial_{t_1} + \sum_{j=2}^{n-m_1-m_2} k_{j} \partial_{t_{j}}\neq 0.
    \end{aligned}
\end{equation}
with, for instance, $\gamma_{1}^{2} - \gamma_{2}^{1} \neq 0$ and $k_1\in \mathbb{R}^{*}$.
From the hypothesis made in \eqref{hp3} it is possible to consider the same choice of polynomials $\P_1$ and $\P_3$ in \eqref{choiceP3} as well as we made in the step two case. This is due to the fact that $\P_3$ does not contain variables with  degree of homogeneity greater than two. Hence when we apply, for instance $X_1$, on $\P_3$ we get
\begin{equation}
    \left[ \sum_{j=1}^{n-m_1-m_2} \left( \sum_{k,i=1}^{m_1} \delta_{kij}^{1}x_ix_k + \sum_{i=1}^{m_2} \theta_{ij}^{1}y_i\right) \partial_{t_{j}} \right] \P_3 = 0.
\end{equation}

Moreover, with the same method we are able to treat the general case  of  a Carnot group of step $s$. In this case, the associated Lie algebra is
\begin{equation}
    \mathfrak{g} = \bigoplus_{i=1}^{s} \mathfrak{g}_{i} = \mathfrak{g}_1 \oplus \mathfrak{g}_2 \oplus  \mathfrak{g}_3 \oplus \bigoplus_{i=4}^{s} \mathfrak{g}_{i}.
\end{equation}
Since we have discussed the problem for a three step Carnot group, from \eqref{1}, we see that the term $T_j$ does not affect the derivatives of $\P_3$ related to the first layer $\mathfrak{g}_1$. Hence it is possible to apply the same procedure even for a general Carnot group.
\end{proof}
\begin{nota}
    In the proof of Theorem \ref{t2} is not only provided a function that realizes the thesis but a family indeed, depending on the parameters $b, p,q$. \\
    The explicit family of functions is, assuming to consider, for instance, the variables $x_1, x_2,\dots,x_{m_1}$ associated with the first stratum and $y$ is one variable associated with the second stratum:
    \begin{equation}\label{polynomialcounterexample}
    \begin{aligned}
        u_{b,p,q}(x_1,x_2,\dots, x_{m_1}, y)&= bx_2 + \frac{\alpha_{1}^{1}(p+q)}{2b(\alpha_{1}^{2} -\alpha_2^1)}x_1^3 + \frac{\alpha_2^1 p + \alpha_{1}^{2} q}{b(\alpha_{1}^{2} - \alpha_2^1)} x_1^2 x_2  +
          \frac{\alpha_{2}^{2}(p+q)}{2b(\alpha_{1}^{2} - \alpha_2^1)}x_1 x_2^2 \\
          & - \frac{q}{3b}x_2^3 - \frac{p+q}{b(\alpha_{1}^{2} -\alpha_2^1)} x_1 \left( y - \sum_{i=3}^{m_1} \left(\frac{\alpha_{i}^{1}x_1}{2} - \alpha_{i}^{2}x_2 \right)x_i \right)
        \end{aligned}
    \end{equation}
    with $b \in \mathbb{R}^{*}, (p,q) \in D.$
\end{nota}

\section{Nonexistence of an ACF formula in some Carnot groups} \label{non}
Given the method of building explicit counterexamples for the functional $\Phi$ shown in the proof of Theorem \ref{t2}, the goal of this section is to apply these results to prove the failure of the increasing monotonicity  of an Alt-Caffarelli-Friedman monotonicity type formula like \eqref{c4} for nontrivial Carnot groups.

\begin{theorem} \label{t3}
     For any Carnot group $\g$ of step $s$, with $s >1$, if the
     fundamental solution $\Gamma_{\mathbb{G}}$ associated with $\Delta_{\mathbb{G}}$ with pole at the origin is symmetric with respect all the variables associated with the first stratum, then there exists a continuous function $u$ such that $u$ is harmonic in $\{u>0\}$ as well as $u$ is harmonic in $\{u\leq 0\}^o$ such that $J_{u}^{\g}$  fails to be increasing in a neighborhood of $0$.
\end{theorem}
\begin{proof} Let us denote, for brevity, 
\begin{align*}
    I_{u^{+}}^{\g}(r):= \frac{1}{r^2} \int_{B_{r}^{\mathbb{G}}} \lvert \nabla_{\mathbb{G}} u^+(M) \rvert^2 \Gamma_{\mathbb{G}}(M) \, dM = \int_{B_{r}^{\mathbb{G}} \cap \{ u >0\}} \lvert \nabla_{\mathbb{G}} u(M) \rvert^2 \Gamma_{\mathbb{G}}(M) \, dM  ,  \\
    I_{u^{-}}^{\g}(r):= \frac{1}{r^2} \int_{B_{r}^{\mathbb{G}}} \lvert \nabla_{\mathbb{G}} u^-(M) \rvert^2 \Gamma_{\mathbb{G}}(M) \, dM =   \int_{B_{r}^{\mathbb{G}} \cap \{ u < 0\}} \lvert \nabla_{\mathbb{G}} u(M) \rvert^2 \Gamma_{\mathbb{G}}(M) \, dM , 
\end{align*}
for which we have that $J_{u}^{\g}(r)= I_{u^{+}}^{\g}(r) \cdot I_{u^{-}}^{\g}(r)$. \\
The goal is to show the monotone decreasing behavior of $J_{u}^{\g}$:
\begin{equation}
    \frac{d}{dr}J_{u}^{\g}(r) = \biggl(\frac{d}{dr}I_{u^{+}}^{\g}(r) \biggl)I_{u^{-}}^{\g}(r) + I_{u^{+}}^{\g}(r) \biggl( \frac{d}{dr}I_{u^{-}}^{\g}(r)\biggl).
\end{equation}
 Apparently, the proof we need has to depend on the special function we have determined only. In fact, the harmonic polynomials we selected in Carnot groups, have a special behavior, in comparison with the companion Euclidean harmonic polynomials.

Since $I_{u^{+}}^{\g}$ and $I_{u^{-}}^{\g}$ are nonnegative, we reach the desired result if we prove that they are both monotone decreasing  or there exists a relationship with the polynomials whose existence we proved in the previous section. 

Indeed, fixing
  $u:=u_{b,p,q}$ provided by Theorem \ref{t2}, we have that $u$ is a difference of two homogeneous polynomials of degree one and three respectively and both intrinsicly harmonic. 

Moreover, recalling  \eqref{polynomialcounterexample}, we point out that
$$ u_{b,p,q}(-x_1,-x_2,\dots, -x_{m_1}, y)=-u_{b,p,q}(x_1,x_2,\dots, x_{m_1}, y).$$

 Hence,
$\mathcal{S}(\{u_{b,p,q}>0\})=\{u_{b,p,q}<0\}$, where $\mathcal{S}$ denotes the change of variables that moves $x_i$ in $-x_i$ for $i=1,\dots,m_1$ and leaves unchanged the other variables. 
Moreover, the hypothesis on the symmetry of $\GammaG$ assure also that $\GammaG(M)=\GammaG(\mathcal{S}(M))$ for every $M \in \g$.

Thus from the construction of $u_{b,p,q}$, as a consequence of the previous remark, recalling the change of variables $\mathcal{S}$, we obtain that
\begin{equation*}
    \begin{aligned}
        I_{u_{b,p,q}^{+}}^{\g}(r)&=\int_{B_{r}^{\mathbb{G}} \cap \{ u >0\}} \lvert \nabla_{\mathbb{G}} u_{b,p,q}(M) \rvert^2 \Gamma_{\mathbb{G}}(M) \, dM = \int_{\mathcal{S}\left(B_{r}^{\mathbb{G}} \cap \{ u > 0\}\right)} \lvert \nabla_{\mathbb{G}} u_{b,p,q}(\mathcal{S}(M)) \rvert^2 \Gamma_{\mathbb{G}}(\mathcal{S}(M)) \, dM  \\
         &=\int_{B_{r}^{\mathbb{G}} \cap \{ u < 0\}} \lvert \nabla_{\mathbb{G}} u_{b,p,q}(M) \rvert^2 \Gamma_{\mathbb{G}}(M) \, dM = I_{u_{b,p,q}^{-}}^{\g}(r)
    \end{aligned}
\end{equation*}
Therefore
$$J_{u_{b,p,q}}^{\g}(r)= I_{u_{b,p,q}^{+}}^{\g}(r) \cdot I_{u_{b,p,q}^{-}}^{\g}(r)=\frac{1}{4}(I_{u_{b,p,q}}^{\g})^2.$$
As a consequence, by the positivity of $I_{u_{b,p,q}}^{\g}$, the decreasing monotonic behavior of $J_{u_{b,p,q}}^{\g}$ descends from the decreasing monotonic behavior of $I_{u_{b,p,q}}^{\g},$ as we proved in  Theorem \ref{t1}.

\end{proof}

We are now in position to obtain the proof of Theorem \ref{t0} as a corollary of Theorem \ref{t3}. \\

\begin{proof}(of Theorem \ref{t0})
The special case of a Carnot group of step two allows us to apply Theorem \ref{t3}. In fact, in the step two case we can prove directly the symmetry of the fundamental solution with respect to variables of the first layer.
This property derives from the integral representation formula of the fundamental solution found in \cite{beals1996green} (Theorem $5.12.1$ in \cite{bonfiglioli2007stratified} in Carnot group settings). Moreover, there exists a isomorphism between every Carnot group of step two with one endowed with inner law governed by skew symmetric matrices (see for instance Proposition $3.5.1$ in \cite{bonfiglioli2007stratified}). This explicit isomorphism fixes the first $m_1$ variables and thus the symmetry of the fundamental solution descends.
In this way we can conclude applying Theorem \ref{t3}.
\end{proof}
$\,$
Concerning Carnot groups of step $s>2$, to our knowledge, neither an explicit representation formula of the fundamental solution is known, in general, nor much is known about the symmetry with respect to the variables of the first stratum. 
Nevertheless it is possible to provide another sufficient condition for the decreasing behavior of $J_{u}^{\g}$ exploiting the well known $\g$-symmetry of $\GammaG$ with respect to the origin, precisely $\Gamma_{\g} (P)=\Gamma_\g(P^{-1}),$ see for instance \cite{Folland_main}. 
Then, in order to state the result described in the following remark, we need to introduce the definition of an \textit{intrinsic odd} function, precisely a function is said to be \textit{intrinsic odd} if $u(P)=-u(P^{-1})$ for every $P \in \g$.
\begin{nota} \label{t4}
     For any Carnot group $\g$ of step $s$, with $s >1$, if it were possible to build a harmonic function $u=\P_1 - \P_3$ such that 
     \[ \ps{\gradG \P_1(P)}{\gradG \P_3(P)} > 0\:\: \text{ a.e. in $\BG{1}{0}$} 
\] 
which is \textit{intrinsic odd} as well, then $u$ would be harmonic in $\{u>0\}$ as well as $u$ is harmonic in $\{u\leq 0\}^o$ and $J_{u}^{\g}$ would fail to be increasing in a neighborhood of $0$.
    This can be seen as another sufficient condition to obtain a counterexample and its proof descends straightforwardly from the proof of Theorem \ref{t3} considering, instead of $\mathcal{S}$, the change of variables that maps $P$ into $P^{-1}$.
\end{nota}

\section{Final remarks}\label{AppEngel}
In this last part, we provide an explicit application of the procedure showed in the case in the first Engel group $\e$, which is a three step Carnot group.

    The first Engel group can be introduced as $\mathbb{R}^{4}$ endowed with the inner law that for every $P,M \in \e$ associate $P \circ M$ in the following way
    \begin{align*}
     P \circ M &= (x_1,x_2,y,t) \circ (x'_1, x'_2,y',t') \\ &= (x_1+x'_1, \, x_2 + x'_2, \,y+ y'+ x_1x'_2, \,t+ t' + x_1y'+ \frac{1}{2}x_1^{2} x'_2). 
     \end{align*}
    A basis of left invariant vector fields is given by
    \[X_1=\partial_{x_1}, \; X_2= \partial_{x_2} + x_1\partial_{y}+ \frac{1}{2}x_1^2 \partial_{t}, \; Y= \partial_{y}+ x_1 \partial_{t}, \; T=\partial_{t}.\]
    The fields $X_1$ and $X_2$ generate the horizontal layer $\mathfrak{g}_1$ of $\e$. The commutator between the vector fields is 
    \[ [X_1, X_2 ]= Y, \quad [X_1, Y]=T,\]
    otherwise is $0.$
    Thus we have
    \begin{equation*}
        \mathfrak{g} = \mathfrak{g}_1 \oplus \mathfrak{g}_2 \oplus \mathfrak{g}_3,
    \end{equation*}
    where 
    \begin{equation*}
         \mathfrak{g}_1 = \Span \{X_1, X_2 \}, \quad \mathfrak{g}_2 = \Span \{Y \}, \quad \mathfrak{g}_3 = \Span \{T \}.
    \end{equation*}
    The sub-Laplacian on $\e$ is defined as
    \begin{align*}  
    \DeltaE  &= (X_1^{2}+X_2^{2}) = \big(\partial_{x_1x_1}^{2} + (\partial_{x_2}+x_1 \partial_{y} + \frac{x_1^2}{2} \partial_{t})^{2} \big) \\
    & = \big( \partial_{x_1x_1}^{2} + \partial_{x_2 x_2}^{2} + x_1^2 \partial_{yy}^{2} + \frac{x_1^4}{4} \partial_{tt}^{2} +  2x_1 \partial_{x_2y}^{2} + x_1^2 \partial_{x_2t}^{2} + x_1^3 \partial_{yt}^{2}\big).
    \end{align*}
Now, applying the procedure of Section \ref{proof} with $p=0, q= \frac{1}{2}, b=1$ by \eqref{coeff} we obtain 
$$ u = \P_1 - \P_3 = x_2 - \biggl(-\frac{1}{2}x_1^2 x_2 + \frac{1}{6}x_2^3 + \frac{1}{2}x_1 y\biggl),$$
for which 
\begin{align*}
\ps{\gradE \P_1}{\gradE \P_3}= \ps{(0,1)}{(\frac{1}{2}y - x_1 x_2, \frac{1}{2}x_2^2)}=\frac{1}{2}x_2^2, \quad \DeltaE u = 0.
\end{align*}
Hence, invoking Theorem \ref{t2} 
we can conclude that $u$ provides a counterexample for the increasing monotone behavior of \eqref{t2}.

The structure of the fundamental solution of the sublaplacian to $\Delta_{\mathbb{E}}$  in the Engel group, to our knowledge, is not explicit. Hence we cannot conclude applying the Theorem \ref{t3}. In fact,  only if  $\Gamma_{\mathbb{E}}$  enjoyed the symmetry requested by Theorem \ref{t3}, we would be in position to produce a counterexample even to the increasing monotonicity of $J_{u}^{\mathbb{E}}$.

\begin{nota}
One could also try to look for harmonic functions being \textit{intrinsic odd} hoping to apply Remark \ref{t4}.
For a general Carnot group is not so simple because the expression of $P^{-1}$ depends on the inner law of the group that could be very complicated.
\end{nota}
In the first Engel group $\mathbb{E}$, for instance, the expression of the inverse of $(x_1,x_2,y,t)$ is given by $(-x_1,-x_2,-y-x_1x_2, -t +x_1y+ \frac{3}{2}x_1^2x_2)$ and it is not possible to build a \textit{intrinsic odd} polynomial of degree three.
However, if we  consider polynomials of degree five it is possible to do so considering $$\P_5 = x_1y^2 - 2yx_1^2x_2 + 2tx_1x_2+\frac{x_1^3x_2^2}{2}+ x_1^2x_2^3$$ then $u= \P_1-\P_5$ with $\P_1=x_2$ is \textit{intrinsic odd} and $\ps{\gradE\P_1}{\gradE \P_5} = 3x_2^2x_1^2>0$ almost everywhere in $B_1^{\e}(0)$ but is not harmonic.  

\section{Data Availability }
Data Availability Data sharing not applicable to this article as no datasets were generated
or analysed during the current study.

\section{Declaration}
Conflict of interest. On behalf of all authors, the corresponding author states that there is
no conflict of interest.

\bibliographystyle{abbrv}
\bibliography{biblio}
\end{document}